\date{Julu 2022}

\documentclass[11pt]{article}
\usepackage{amsthm,amscd}
\usepackage{enumerate}
\usepackage{tikz}
\usepackage{lipsum}
\usepackage{lineno}
\usepackage{caption}
\usepackage{float}

\usepackage[utf8]{inputenc}
\usepackage{color}
\usepackage{subfigure}
\usepackage{url}
\usepackage{graphicx}
\usepackage{amsmath}
\usepackage{amsthm}
\usepackage{bm}
\usepackage[plainpages=false]{hyperref}
\usepackage{enumerate}
\usepackage{datetime}

\usepackage{titlesec}
\titleformat{\section}
	{\normalfont\Large\bfseries\filcenter}{\thesection.}{1 ex}{}
\titleformat{\subsection}
	{\normalfont\normalsize\bfseries}{\thesubsection.}{1 ex}{}
\titleformat{\subsubsection}[runin]
	{\normalfont\normalsize\bfseries\filcenter}{\thesubsubsection.}{1 ex}{}

\usepackage[charter]{mathdesign}
\usepackage[mathcal]{eucal}

\usepackage[margin=1.2 in]{geometry}

\usepackage[square,numbers,sort&compress]{natbib}
\newcommand*{\doi}[1]{doi: \href{https://dx.doi.org/#1}{\urlstyle{rm}\nolinkurl{#1}}}
\newcommand*{\arxiv}[1]{arXiv:  \href{https://arxiv.org/abs/#1}{\urlstyle{rm}\nolinkurl{#1}}}

\usepackage{thmtools,thm-restate}
\declaretheorem[within=section]{theorem}
\declaretheorem[sibling=theorem]{lemma}

\declaretheorem[sibling=theorem]{corollary}
\declaretheorem[sibling=theorem]{conjecture}

\newcommand\RR{\mathbb{R}}
\newcommand\PP{\mathbb{P}}

\newcommand{\defn}[1]{\emph{\color{blue} #1}} 

\newcommand\p{{\bf p}}

\newcommand\rr{{\bf r}}

\newcommand\q{{\bf q}}

\title{Universal rigidity of ladders on the line}

\author{Bryan Chen\thanks{Partially supported by NSF grant PHY-1554887}, Robert Connelly\thanks{Department of Mathematics, Cornell University. \texttt{rc46@cornell.edu}. Partially supported by NSF grant DMS-1564493}, Steven J. Gortler\thanks{School of Engineering and Applied Sciences, Harvard University, Cambridge, USA. Partially supported by NSF Grant:XXX and the Aalto Science Institute (AScI) Thematic Program 
``Challenges in large geometric structures and big data''.}, Anthony Nixon\thanks{Department of Mathematics and Statistics, Lancaster University. \texttt{a.nixon@lancaster.ac.uk}.}, and Louis Theran\thanks{School of Mathematics and Statistics, University of St Andrews. \texttt{lst6@st-and.ac.uk}}}

\begin{document}
\date{July 2022}
\maketitle

\begin{abstract}
In \cite{Line-rigidity} it is shown, answering a question of Jord\'an and Nguyen \cite{JN15}, that universal rigidity of a generic bar-joint framework in $\RR^1$ depends on more than the ordering of the vertices.   The graph $G$ that was used in that paper is a ladder with three rungs.  Here we provide a general answer when that ladder with three rungs in the line is universally rigid and when it is not. 

\end{abstract}

\section{Introduction}

A framework is a structure composed of stiff bars linked at universal joints.
Here we investigate the universal rigidity of a particular framework, called a ladder, in the line $\RR^1$.  We give a complete description of when it is rigid in all higher dimensions (one definition of universal rigidity) and when it is not rigid in higher dimensions.  This was motivated by a question posed by Tibor Jord\'an and V.-H. Nguyen in \cite{JN15} as to whether universal rigidity in $\RR^1$ only depends on the order of its vertices.  In \cite{Line-rigidity} it is shown that the ladder at two particular configurations in the line is such that the vertices are in the same order, but one is universally rigid and the other is not universally rigid. Here we give a complete description of when any configuration of the ladder in the line, with no zero length edges, is universally rigid.  It is universally rigid when a particular polynomial, that we call the discriminant, is positive in the $6$ coordinates of the ladder, and it is flexible in higher dimensions when the discriminant is negative. When the discriminant is zero, it can be either universally rigid or not.  The argument here uses techniques that can be applied in a much larger context to show both universal and non-universal rigidity.  In particular, our argument builds on results in \cite{iterative} that give a general condition when a framework in any dimension is universally rigid.

\section{Definitions and set-up}\label{sect:definitions}

A \defn{framework} $(G,\p)$ is an ordered pair consisting of a graph $G$ with $n$ vertices and a \defn{configuration} $\p = (p_1, \ldots, p_n)$, with each $p_i \in \RR^d$ corresponding to vertex $i$ in the graph $G$. We say that $(G,\p)$ is \defn{universally rigid} in $\RR^d$ if any other framework $(G,\q)$ with corresponding edge lengths the same as in $(G,\p)$ and $\q$ a configuration in $\RR^D$, for some $d\le D$, is such that $\q$ is congruent to $\p$.

An \defn{equilibrium stress} or simply a \defn{stress} is a scalar $\omega_{ij}$ assigned to each edge $ij$ corresponding to each pair of  vertices  of $G$ such that:
\begin{equation}\label{eq: stress}
    \sum_{j\neq i} \omega_{ij} (p_i-p_j)=0.
\end{equation}
Each non-edge of $G$ has $\omega_{ij}=0$. To each such stress we define a \defn{stress matrix} $\Omega=\Omega(\omega)$ whose $i,j$ entry, for $i \ne j$, is $-\omega_{ij}$ and whose diagonal entries are such that the entries in each row, and in each column, sum to $0$.  

We say $(G,\p)$ is \defn{super stable} in $\mathbb{R}^d$ if there exists an equilibrium stress $\omega$ for $(G,\p)$ 
such that $\Omega(\omega)$ is positive semi-definite (PSD) of rank $n-d-1$, where $n$ is the number of vertices of $G$, and the edge directions of $(G,\p)$ do not lie on a conic at infinity. A framework $(G,\p)$ in $\RR^d$ is defined to lie on a \defn{conic at infinity} if there is  a non-zero $d$-by-$d$ symmetric matrix $Q$ such that in $G$ for each edge $i,j$ between vertex $i$ and vertex $j$, $(p_i-p_j)^tQ(p_i-p_j)=0$.  Note that $Q$ is necessarily not PSD, and for $d=2$, the set of such directions is just two directions, that is just two points on the (projective) line at infinity.

The following \cite{C83} is fundamental. 

\begin{theorem}\label{thm:super}
Let $(G,\p)$ be a super-stable framework in $\mathbb{R}^d$. Then $(G,\p)$ is universally rigid in $\mathbb{R}^d$.
\end{theorem}

In the generic case the converse is also true (when the number of vertices is at least $d+2$) by Gortler and Thurston in \cite{GT14}. This equivalence breaks down without the genericity hypothesis (for example, see Figure \ref{Projections.fig}). The desire to understand the non-generic case motivates us to study the boundary between universally rigid and non-universally rigid frameworks.

For most frameworks $(G,\p)$, 
we have observed that when it is universally rigid, it is the limit of super stable frameworks $(G,\q)$.  One simple exception is in the line when $G$ consists of just one bar, and in the plane when the vertices of $\p$ lie in two lines and has a bar connecting a vertex in one line to a vertex in the other line.  The stress is $0$ on that bar, and the bar is needed for universal rigidity.

A natural situation is to extend the theory of bar frameworks to tensegrities, where each member, corresponding to an edge of the graph $G$, is either a strut (constrained not to get shorter) or a cable (constrained not to get longer).  A bar is both a strut and a cable.  With this in mind we propose the following: 

\begin{conjecture} \label{Question} If a tensegrity $(G,\p)$  in $\RR^d$, with each member either a strut or a cable, is universally rigid, then it is the limit of super stable tensegrities $(G,\q)$.  
\end{conjecture}

We have no counterexamples to this conjecture.

In the following we take a closer look at a particular graph, the ladder, defined in the next section, and answer Conjecture \ref{Question} just for that graph.  All universally rigid frameworks of the ladder in the line are limits of super stable frameworks. 

\section{ The Ladder}\label{Ladder}

We now take a closer look at a particular framework in the plane that was central in the examples of \cite{Line-rigidity}, the Dessargues' configuration, and the three-rung ladder described in Figure \ref{ladders.fig}, which is especially relevant when they appear in the line. In all these we assume that all the struts, cables, and bars, have non-zero length.
\begin{figure}[ht]
\centering 
\includegraphics[scale=0.5]{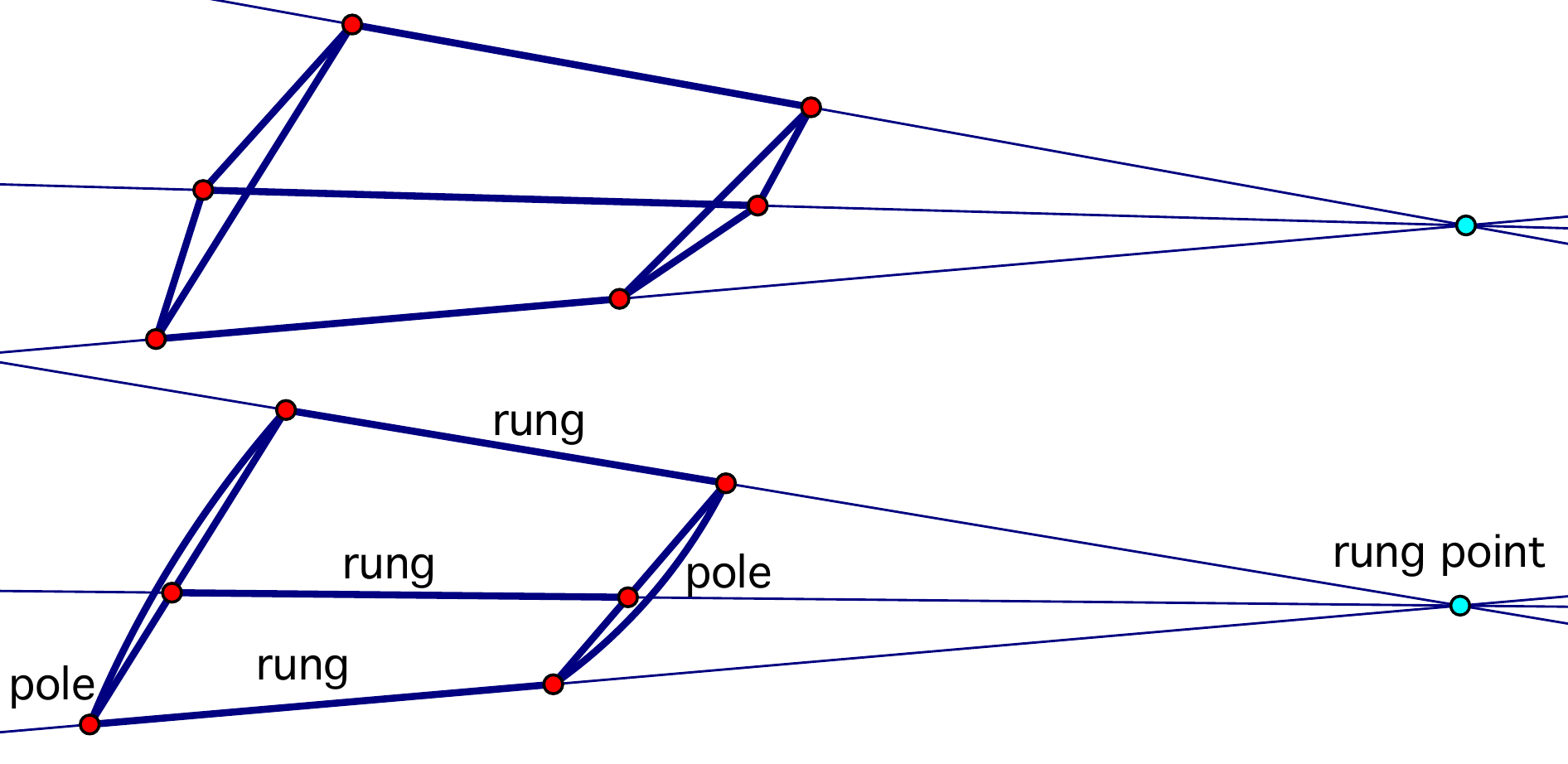}
\captionsetup{labelsep=colon,margin=1.3cm}
\caption{These are two frameworks in $\RR^2$. The top one is the Desargues configuration, where the three lines connecting the corresponding vertices of the two triangles meet at point.  The bottom one is the three-rung ladder. The three lines through the rungs also meet at a point, in each case, possibly at infinity, i.e. parallel. In the poles, we have indicated the bar between its end points with a curved line.}
\label{ladders.fig}
\end{figure}

The \defn{Desargues graph} consists of $6$ points, $A,B,C,a,b,c$, where $A,B,C$ and $a,b,c$ each form a triangle while $A,a$; $B,b$; and $C,c$ are three additional edges.  The \defn{Desargues framework} is a corresponding configuration where we assume that none of the edges have zero length and the three lines determined by the edges $A,a$; $B,b$; and $C,c$ meet at a single point, possibly at infinity in the associated projective plane.  The \defn{ladder} is a framework whose graph is the same as is a Desargues graph, but where  $A,B,C$ lies in a line and $a,b,c$ lies in a line.  Note that when the corresponding configuration lies in $\RR^1$, the Desargues framework is automatically a ladder.

With respect to the Desargues framework in Figure \ref{ladders.fig} we have the following:

\begin{theorem}[{\cite[Theorem 5]{C83}}] \label{thm:con83}
Let $(G,\p)$ be a framework/tensegrity in $\mathbb{R}^2$, so that the vertices of  $\p$   are in strictly convex position
and every edge of the convex hull of  $\p$   is an edge of $G$; we call these  
boundary edges cables and the rest of the  interior edges, struts.
If $\omega$ is an equilibrium stress of $(G,\p)$ that is positive 
on the boundary edges and negative on the interior edges, then 
the associated stress matrix is PSD (positive semi-definite) of rank $n - 3$ and $(G,p)$ is super stable.
\end{theorem}

Note that by \cite{iterative} the non-singular image projective image of a Desargues framework in convex position as in Theorem \ref{thm:con83} is super stable.  We call any of these frameworks \defn{super stable Desargues frameworks}. Note that when the rung lines do not intersect at a point, the  framework is not universally rigid, or when they do intersect at a point, it is still possible for the Desargues framework not to be even universally rigid. It depends on the configuration.

Given a ladder with a two dimensional span, 
if the three rungs of a  ladder
are parallel and the two poles are parallel
we call it a \defn{step ladder}.
A step ladder flexes smoothly from the plane 
down to the line.

Given a ladder with a two dimensional span, 
if all the lines through the rungs are incident to a point, possibly at infinity, and all the pole lines and rung lines are not in just in two directions,
we call it
an \defn{orchard ladder}.
An orchard ladder is universally rigid but not
super stable~\cite{iterative}.

In order to relate this to frameworks in the line we have the following from \cite{Line-rigidity}.

\begin{theorem}\label{thm:universal}  If $(G,\hat \p)$ is a framework  that is 
universally rigid in $\RR^D$,  then the affine image 
of $(G,\hat \p)$ onto $(G,\p)$ in $\RR^d$, $D\ge d$, is universally rigid in $\RR^d$.
\end{theorem}
In \cite{Line-rigidity} Theorem \ref{thm:universal} is only shown for orthogonal projections, but it is easy to see that any affine map can be obtained by a composition of orthogonal projections and scaling.

We next  notice what happens when we have a 
super stable Desargues framework as in Figure \ref{ladders.fig}.

\begin{lemma}\label{lem:max-rank} If a framework is a Desargues framework that is in equilibrium in the plane as a convex hexagon with the three diagonals as in Figure \ref{ladders.fig}, then it and any orthogonal projection into a line are super stable.
\end{lemma}
\begin{proof} By Theorem \ref{thm:con83} the Desargues configuration is super stable.  Its PSD stress serves as a PSD stress for its  orthogonal projection into the line.  The PSD stress on each triangle in the projection also serves as a PSD stress of the projection.  The sum of those two stresses is a maximum rank stress for the projected configuration in the line, since any other configuration in the line with the sum of those stresses has to be an affine image of both, and that is only possible when it is an affine image of the whole line, and thus super stable. This is the same argument used in \cite{Line-rigidity}.\qedsymbol 
\end{proof}

\section{ Coordinate System}\label{Coordinate-System}

The three-rung ladder has two poles, and two co-linear triangles. One co-linear triangle is at coordinates $A,B,C$ and another is at coordinates $a,b,c$. The three rungs  connect $A,a$, and $B,b$, and $C,c$. We assume that these numbers are such that in $\RR^1$ all the members (edges) have non-zero length. Unless otherwise stated, we will assume that we have an orchard ladder, where the points of the first pole are $A,B,C$ in the $x$ axis, and there are  points $a',b',c'$ in the plane projecting orthogonally onto $a,b,c$ respectively. 
 
\begin{figure}[ht]
\centering 
\includegraphics[scale=0.6]{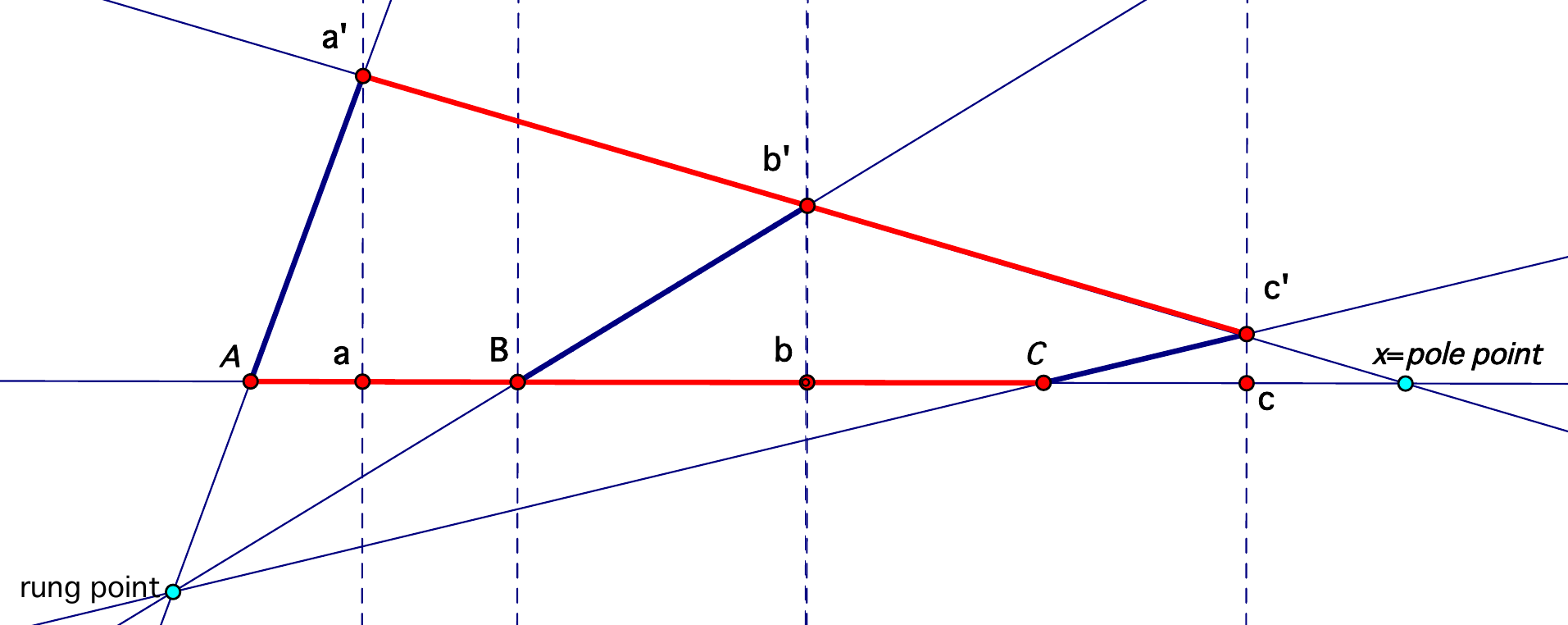}
\captionsetup{labelsep=colon,margin=1.3cm}
\caption{This shows an orchard ladder projected orthogonally into the line through one of its poles. The dashed lines show orthogonal projection onto the horizontal pole as the $x$-axis.  The thin solid lines show projection from the rung point, the intersection of the lines through the three rungs of the ladder. The point $x$ is the intersection of the two poles, the pole point. All the points $A,B,C,a,b,c,x$ are regarded as real numbers on the horizontal axis as well as points in the plane. The poles are in red.}
\label{Projections.fig}
\end{figure}

Imagine that the six points $A,B,C,a,b,c$ are given and we wish to construct the orchard ladder as in Figure \ref{Projections.fig}.  The following construction attempts to do this.

\begin{subsection}{A Construction}\label{SubSect:construction}

Construct the three dashed lines through $a,b,c$ perpendicular to the horizontal pole as in  Figure \ref{Projections.fig}.  
\begin{itemize}
    \item Choose $a'$ anywhere on its dashed line, but not on the horizontal pole, as in the figure.  The success of the construction does not depend on this choice of $a'$.  
    \item Choose the point $x$ on the horizontal pole, but allow it to vary.  Then all the rest of the construction is determined as follows.
    \item The line through $a'$ and  $x$ and the vertical line through $c$  determines ${c'}$.  
    \item The lines $C, c'$ and $A,a'$ determine the rung point. 
    \item The line through the rung point and $B$, and the pole line through $a'$ and $c'$, determine the point $b'$.  
    \item If we have chosen $x$ well and $b'$ is on the vertical dashed line through $b$, we have an orchard ladder projecting orthogonally onto our framework in the line.  There may not be such an orchard ladder and every choice of $x$ will fail. Furthermore, if $x$ is chosen at infinity (i.e. the pole lines are parallel), and the rung point is also at infinity, then, instead of an orchard ladder, we get a step ladder.  We will see later how to choose $x$ properly.
    \end{itemize}
\end{subsection}

From this picture one can see that the points $A,B,C$ project to the points $a',b',c'$, in order, from the rung point.  Then $a',b',c'$, in order, project orthogonally onto $a,b,c$.  Throughout both of these projections, the point $x$ is fixed.

\section{Cross-ratio}\label{Cross-Ratio}

We recall from projective geometry the notion of a cross-ratio.  Let $z_1, z_2, z_3, z_4$ be four points on the real line, where we assume that each of the four terms below are distinct.  But one may be "at infinity", in the sense of the real projective line.  Their \defn{cross-ratio} is defined as
\[
\frac{(z_3-z_1)(z_4-z_2)}{(z_3-z_2)(z_4-z_1)}.
\]
On the projective line, if, 
say, $z_4$ is at infinity, then the cross ratio becomes
\[
\frac{(z_3-z_1)}{(z_3-z_2)}.
\]
The following is a well-known result, see \cite{Coxeter}, where $\PP^1$ is the real projective line regarded as the real line $\RR^1$ with one additional point at  infinity:
\begin{theorem}\label{thm:orchard}
If $f: \PP^1 \rightarrow \PP^1 $ is a projective map between real projective lines, and $f(z_i)=w_i$ for $i=1,2,3,4$, where all the  factors below are non-zero, then the cross ratios are preserved. When all points are finite,
we get
\[
\frac{(z_3-z_1)(z_4-z_2)}{(z_3-z_2)(z_4-z_1)}=\frac{(w_3-w_1)(w_4-z_2)}{(w_3-z_2)(w_4-w_1)}.
\]
\end{theorem}

From the discussion in Section \ref{Coordinate-System}, we see that the points $x,A,B,C$ are sent to the points $x,a,b,c$, in order, by a series of two projections, when there is such an ladder projecting onto the framework in the line. Suitable care is to be taken when one of the points is at infinity.  So we get the following:

\begin{theorem}\label{thm:ratio} There is a ladder projecting onto the points $A,B,C,a,b,c$, as in Figure \ref{Projections.fig}, if and only if the following quadratic equation in finite $x$ holds:
\begin{equation}\label{eq: ratio}
(B-x)(C-A)(b-a)(c-x)=(b-x)(c-a)(B-A)(C-x),
\end{equation}
or equivalently,
\begin{equation*}
\frac{(B-x)(C-A)}{(B-A)(C-x)}=\frac{(b-x)(c-a)}{(b-a)(c-x)}.
\end{equation*}
or 
$\frac{(C-A)}{(B-A)}=\frac{(c-a)}{(b-a)}$ and $x$ is at infinity.
\end{theorem}
The case with $x$ at infinity is
just a solution to 
the projective version of Equation~\ref{eq: ratio}.

\begin{proof} If there is an orchard ladder projecting onto the points $A,B,C,a,b,c$, the discussion above shows that the equation (\ref{eq: ratio}) holds. 

Conversely suppose equation (\ref{eq: ratio}) holds.  Use the sequence of projective constructions in Subsection \ref{SubSect:construction} to start the construction of the orchard ladder.  Any value of $x$ that satisfies Equation (\ref{eq: ratio}) must be the point that is in the last projection that is determined by the other five points. Since the orchard ladder is universally rigid, and it projects orthogonally into the line, by Theorem \ref{thm:universal}, it is universally rigid. \qedsymbol
\end{proof}

\section{Discriminant}\label{Sect:discriminant}

The basic problem is given the $6$ coordinates of the ladder graph on the line, determine whether that configuration in $\RR^1$ is universally rigid, is super stable, is flexible from $\RR^1$ into $\RR^3$ and then down into another realization in $\RR^2$, or flexes from $\RR^1$ as a step ladder and just flexes in $\RR^2$.  A crucial number to calculate is the discriminant of the quadratic equation (\ref{eq: ratio}).  This is a homogeneous polynomial in the $6$ variables with 85 terms $D(A,B,C,a,b,c)$.  The coefficients of the $x^2$, $x$, and constant terms from the polynomial in equation (\ref{eq: ratio}) are respectively:
\begin{eqnarray}\label{eqn:dis}
\mathcal{A}&=&(C-A)(b-a)-(c-a)(B-A)\label{eqn:dis-1}\\
\mathcal{B}&=&-(B+c)(C-A)(b-a)+(b+C)(c-a)(B-A)\label{eqn:dis-2}\\
\mathcal{C}&=&cB(C-A)(b-a)-Cb(c-a)(B-A).\label{eqn:dis-3}
\end{eqnarray}
Since all of the edge lengths are non-zero, one can 
verify that at least one
of $\mathcal{A}$, $\mathcal{B}$, or $\mathcal{C}$ is non-zero.
The discriminant of this quadratic is
\[
D=D(A,B,C,a,b,c)=\mathcal{B}^2-4\mathcal{AC}.
\]

When 
$D(A,B,C,a,b,c) > 0$, there are two projective
solutions, when $D=0$ there is one, and
when $D<0$ none.
By \ref{thm:ratio} each solution corresponds
to a lift that is a ladder in $\RR^2$.
Meanwhile, when $\mathcal{A}=0$, then one
of the solutions has $x$ at infinity, in which
case the poles of the lift are parallel.

\begin{theorem} \label{thm:newRatio}
Let $(G,\p)$ be a framework of the 
ladder graph in $\RR^1$.
Then there exists an ladder  in $\RR^2$ projecting
down to $(A,B,C,a,b,c)$
if and only if $D(A,B,C,a,b,c) \ge 0$.
If $D > 0$ this must be an orchard ladder.
\end{theorem}
\begin{proof}
The first part of the theorem is proven
by the above discussion, using Theorem~\ref{thm:ratio}.

Next we argue that get a step ladder, 
if and only if we 
have $\mathcal{A}=0$ and $\mathcal{B}=0$.
This will imply that $D=0$, as stated 
in the theorem.

To get an step ladder, we have  parallel
poles, which means that $\mathcal{A}=0$.
This implies  
$(C-A)(b-a)=(c-a)(B-A)$.
In this case
$\mathcal{B} = 0$ if and only if   
$B+c = b+C$ from equation (\ref{eqn:dis-3}). 
This is equivalent to, $B-b = C-c$.
This is together with parallel poles is equivalent
to parallel rungs, giving us a step ladder. 
\end{proof}

\begin{lemma}\label{lemma:projection} Suppose $(G,\p)$ is a  framework of the ladder graph in $\RR^1$ with all edges of non-zero length, and its discriminant $D(\p)=0$. Then it has a lift
to a ladder framework
$(G,\q)$ in $\RR^2$, from the construction in Section \ref{SubSect:construction}, that projects orthogonally onto it, where the following holds: 
\begin{enumerate}
    \item The rung point, pole point, and projection point (on the line at infinity) of $(G,\q)$ onto $(G,\p)$ are collinear as in Figure \ref{Projections-3.fig},
    \item $(G,\p)$ is not super stable,
    \item $(G,\p)$ is universally rigid if and only if the rungs of $(G,\q)$ are not parallel, or the poles of $(G,\q)$ are not parallel.
 \end{enumerate} \label{lem:D=0}
 \end{lemma}
 
\begin {proof} 
From Theorem~\ref{thm:newRatio}, $(G,\p)$ must lift to a ladder
$(G,\q)$.

One useful special case is when the coefficient $\mathcal{A}=0$ in (\ref{eqn:dis}) and $D=0$. As described in the proof of
Theorem~\ref{thm:newRatio}, the lift is a step ladder.
This shows that Condition (1.) is true in this case.  
When a projective transformation takes the pole point $x$ back to a more general value, Condition (1.) must still be true.
Going the other way, 
the property that $D=0$ is preserved under 
projective transformations of the line, 
and we can find
such a projective transformation taking $x$ to infinity,
so that $\mathcal{A}=0$.

Condition (2.): Clearly when $x$ is at infinity, $(G,\p)$ does not have a PSD stress matrix of maximal rank $6-1-1=4$, since if it did, it would be universally rigid then, and it is not since it would then be a step ladder.  (There is no conic at infinity to worry about in the line.)  But under a projective transformation, the rank and PSD nature of any stress matrix are preserved.  Thus any projective image will not be super stable was well, and Condition (2.) holds.

Condition (3.) follows from Theorem \ref{thm:ratio} and Theorem \ref{thm:universal}. (The lift in the plane is an orchard ladder and is thus universally rigid.) 
\qedsymbol
\end{proof}

We have the following:
\begin{theorem}\label{discriminant:thm}  
Let $(G,\p)$ be a framework of the 
ladder graph in $\RR^1$, with
 $D(A,B,C,a,b,c)>0$.
Then there exists a two-dimensional super stable Desargues framework $(G,\q)$ in $\RR^2$
that projects to $(G,\p)$. 
Thus,  
$(G,\p)$ itself is super stable.
\end{theorem}
\begin{proof} 
Since $D > 0$, from Theorem~\ref{thm:newRatio}, there exists an orchard ladder configuration $(G,\rr)$ in $\RR^2$ that projects onto $(G,\p)$.  
Call the point at infinity $p_{\infty}$ that is used to project the ladder into the line.  
Since $D(A,B,C,a,b,c) \neq 0$, then $p_{\infty}$, the pole point $x$, and rung point are not collinear by  Lemma \ref{lemma:projection}.  
We next show that there is a super stable configuration $\q$ (that can be chosen arbitrarily close to $\rr$ if desired) that will also project from $p_{\infty}$ into the line $L$ containing the ladder.  Note that the object is to choose $\q$ such that each vertex $q_i\ne p_{\infty}$ of $\q$ is on the line through  $p_{\infty}$ and $p_i$.  Once this is done, it does not matter where the line $L$ is that we are projecting onto.  So, for convenience, we can take a non-singular projective image of the configuration to help see what is going on.  In particular, we can take the pole point $x$ and the rung point to be on the line at infinity, but  the projection point $p_{\infty}$ is NOT on the line at infinity.  So we get a situation, where one case is shown in Figure \ref{fig:ladder-perturb}.

\begin{figure}[ht]
\centering 
\includegraphics[scale=0.25]{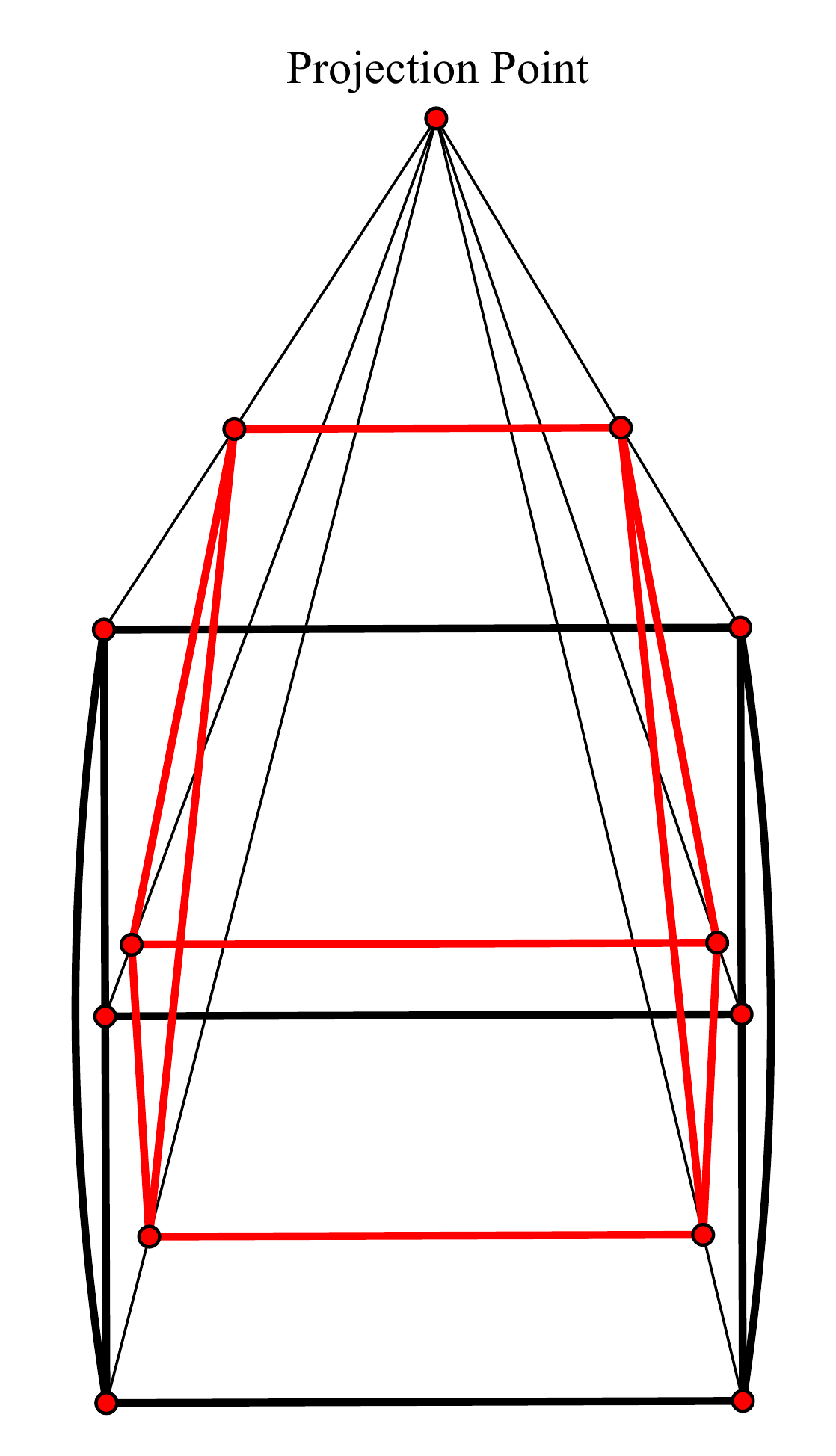}
\captionsetup{labelsep=colon,margin=1.3cm}
\caption{This shows a perturbation of a ladder in black, in the plane to get a super stable Desargues configuration, in red, that projects from the Projection Point, $p_{\infty}$, shown,  onto $L$, not shown.  }
\label{fig:ladder-perturb}
\end{figure}

Here $p_{\infty}$ is chosen ``between" the two pole lines, and it is easy to see how to do the perturbation. The rung point does not have to change.  The super stability of the projection into the line follows from Lemma \ref{lem:max-rank}.   When $p_{\infty}$ is outside that region, it is a bit harder to see how to do the perturbation.

Next we consider the case when $p_{\infty}$ is outside the region between the two pole lines.  
\label{ladder-perturb.fig}
In Figure \ref{rung-motion.fig} we have shown the final Desargues configuration that projects onto the line. Without loss of generality, we have assumed that the pole point is in right on the $A,B,C$ pole. Since we assume the rung point is on the $A,a'$ line (by definition) and outside the region between the two poles, it might as well  be down and to the left.  There is no loss in generality in assuming that the $A,a'$ slope is positive as in the picture. In the perturbed picture, the rung point is moved down and to the left, while the $a'$ point moves down to $a''$ and $c''$ moves up from $c'$, but not as much as $a'$ moves down.  Thus $a'', b'', c''$ forms a triangle outside of the $A'', a'', c'', C$ quadrilateral, as does the $A'',B, C$ triangle. Thus the whole configuration forms a convex hexagon and is a super stable Desargues configuration.  Thus by Lemma \ref{lem:max-rank}, again, the projection into the line is super stable.
 \qedsymbol
 
 Notice also, with this argument, we do not have to assume that the configuration is generic, only that $D(A,B,C,a,b,c)>0$ to insure that, at the  configuration, the framework is super stable.

\begin{figure}[ht]
\centering 
\includegraphics[scale=0.55]{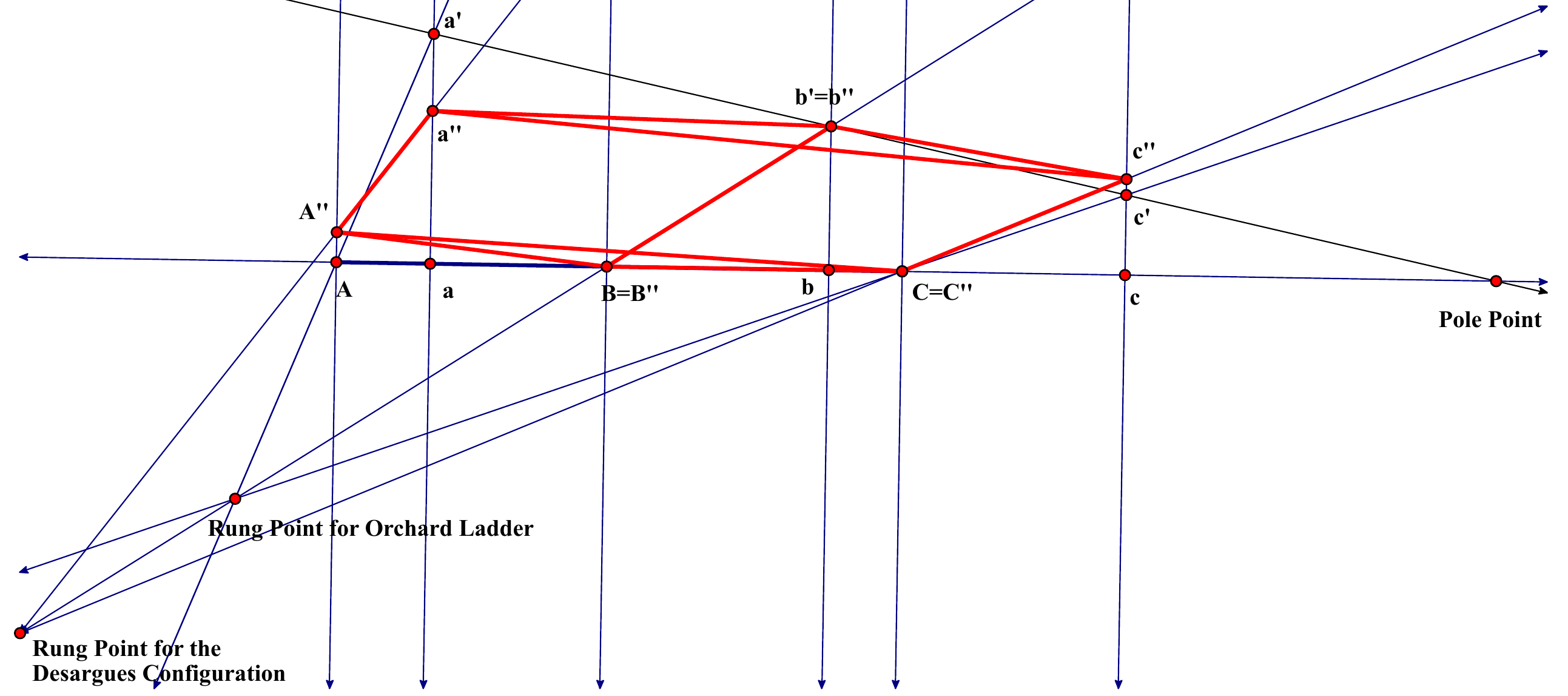}
\captionsetup{labelsep=colon,margin=1.3cm}
\caption{This shows a perturbation of the  orchard ladder in Figure \ref{Projections.fig} to get a PDF Desargues configuration projecting orthogonally into $\RR^1$.  $A,B,C,a,b,c$ in the line is replaced by Desargues configuration $A'',B,C,a'',b'',c''$ in the plane that projects onto it.}
\label{rung-motion.fig}
\end{figure}
\end{proof}

\section{Example}\label{Sect:Example}
Later in the paper, we will see how the negative discriminant implies that the structure is not universally rigid.  
In this section, we describe an explicit example, at the boundary of configurations that are universally rigid and not universally rigid.
The Framework where $(A,B,C,a,b,c)=(1,3,5,2,4,6)$ is an example of a step ladder, which is flexible in the plane. It is \defn{dimensionally rigid} in the sense of Alfakih in \cite{Alfakih-dimension}, where the framework, with given edge lengths, exists in a particular maximal dimension, which, in this case, is dimension $2$. It turns out that the step ladder is at the boundary of the universally rigid and non-universally rigid configurations, and, of course, itself is not universally rigid.  So we consider those configurations of the form $(A,B,C,a,b,c)=(1,3,5,2,4,c)$, for varying $c$.  The values of $c$, where the discriminant $D(1,3,5,2,4,c)\ge 0$ is a connected closed subset of the real projective line $\PP^1$, namely the closed interval between $10/7$ and $6$.
The end points are not super stable.  We calculate that
\[
D(1,3,5,2,4,c)=-7c^2+52c-60=(10-7c)(c - 6).
\]
So $D(1,3,5,2,4,c) >0$ if and only if $10/7<c<6$. So when $c=10/7$ we get the  configuration in Figure \ref{Projections-3.fig} in the plane projecting orthogonally onto the one in $\RR^1$, where $x=5/2$.
\begin{figure}[ht]
\centering 
\includegraphics[scale=0.6]{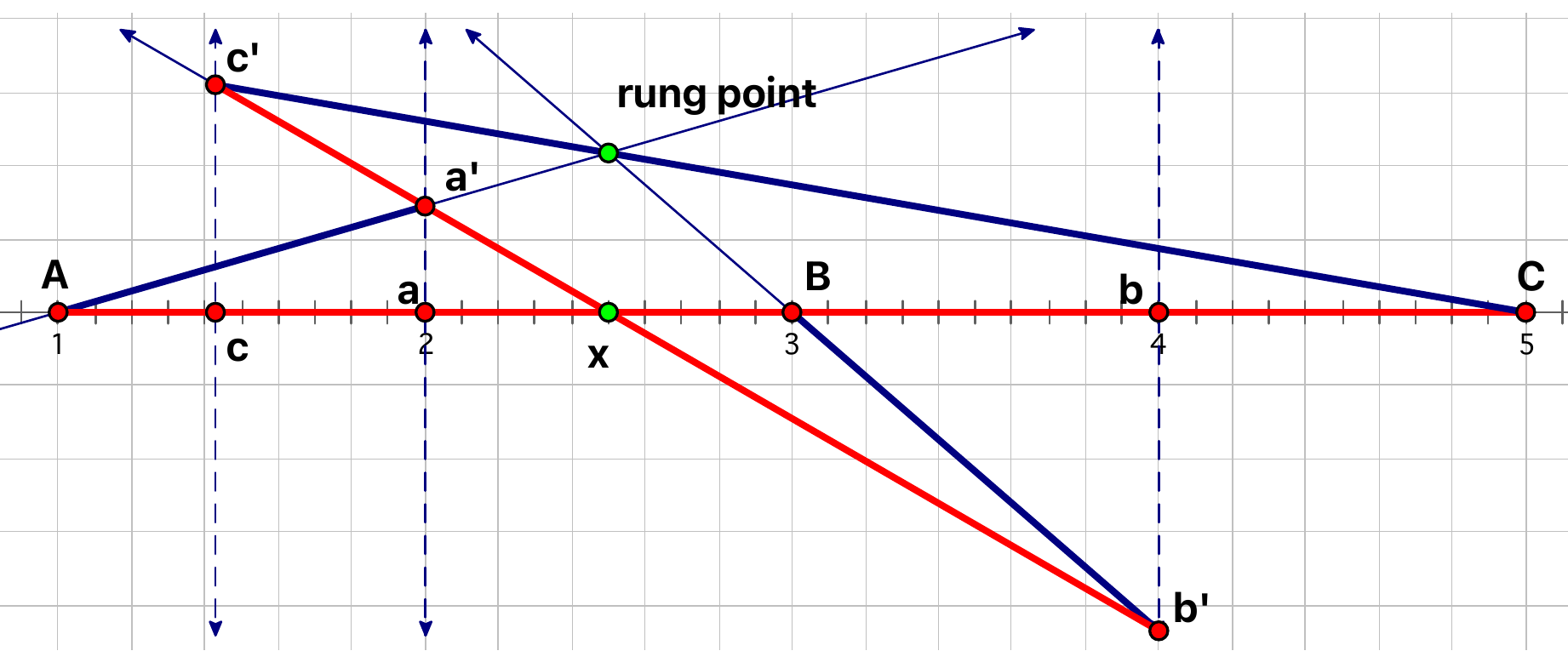}
\captionsetup{labelsep=colon,margin=1.4cm}
\caption{This shows a critical orchard ladder  in the plane which is universally rigid, so its projection in the line is also universally rigid. But after a small perturbation one way in the line, it becomes not universally rigid, and after a small perturbation the other way it becomes super stable in the line.  (The red members are the poles.  See Figure \ref{Tensegrity-ladders.fig}(e), below.) The point $x$ is not one of the vertices of the framework.  It is the intersection of the two poles.}
\label{Projections-3.fig}
\end{figure}

One can calculate independently that when $c<10/7$ (and all the other parameters the same), then the configuration in the
line flexes in $\RR^3$ then back down into the plane in a different configuration.

\section{Tensegrities}\label{Sect:Tensegrity}

First we notice that when the graph $G$ is the ladder graph, and  the bar framework $(G,\p)$ is super stable, in $\RR^2$, with no zero length members, then the rung lines intersect at a point.  Furthermore, it has a stronger property, where $(G,\p)$ is super stable as a tensegrity.  This is where some of the edges, designated as \defn{cables} (drawn with dashed lines) are only required to not increase in length, and other edges designated as \defn{struts} (drawn with thick solid lines) are only required to not decrease in length.  Figure \ref{Tensegrity.fig}  shows some examples of such super stable Desargues tensegrities.

\begin{figure}[ht]
\centering 
\includegraphics[scale=0.7]{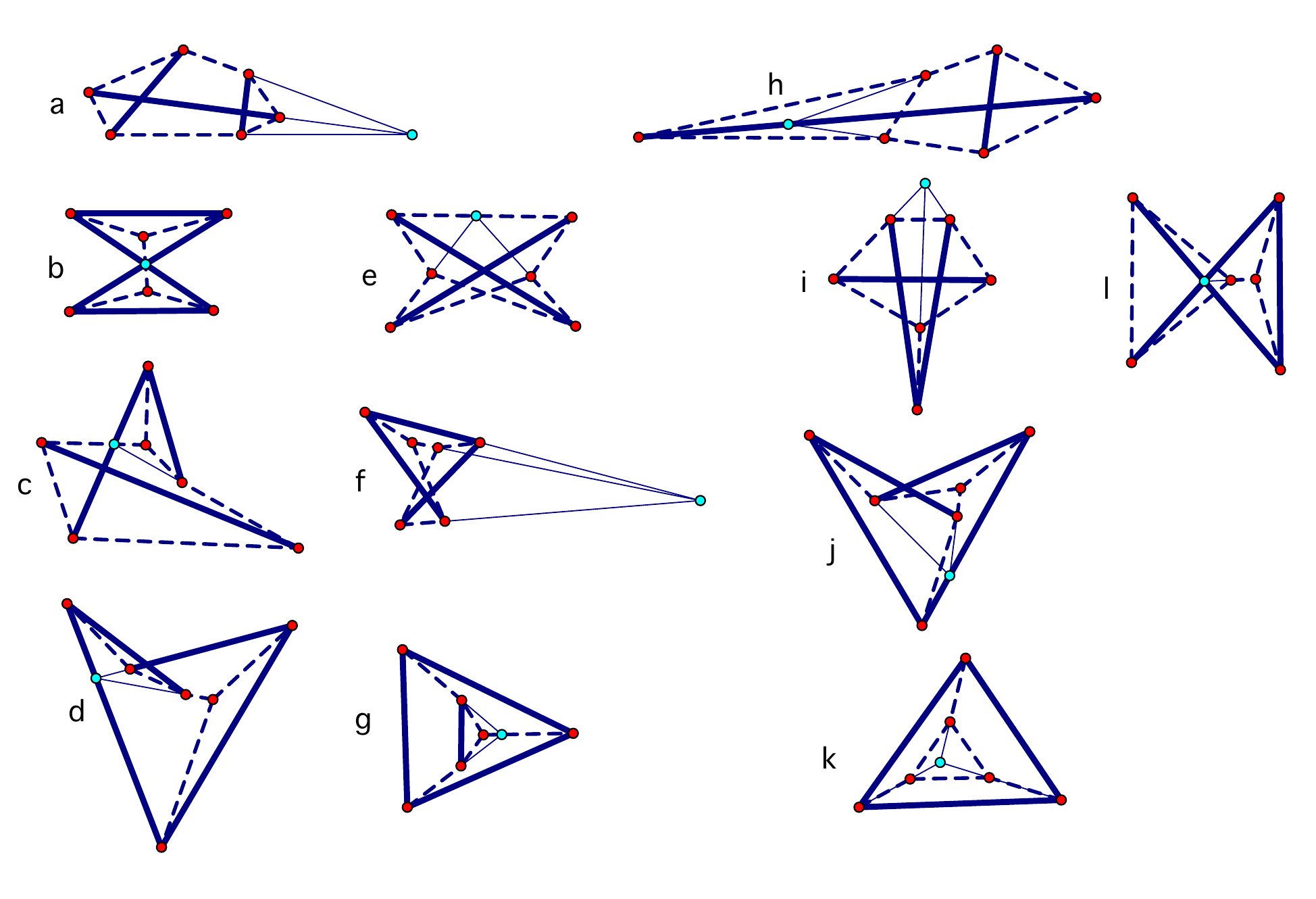}
\captionsetup{labelsep=colon,margin=1.3cm}
\caption{This shows super stable tensegrity frameworks of the Desargues graph in various configurations.  Solid heavy edges are struts, and dashed edges are cables.  Vertices are red points, while the blue points show where the extensions of the rung lines, in the thin solid lines, intersect.  These are all the possibilities for strut/cable patterns for the universally rigid Desargues graph.  There are two projective equivalence classes, {a \dots g} and h \dots k.  See the Addendum.}
\label{Tensegrity.fig}
\end{figure}

When the triangles in Figure \ref{Tensegrity.fig} are such that they  lie in a line, we get the  configurations of tensegrity ladder configurations in Figure \ref{Tensegrity-ladders.fig}.  For example, Figure \ref{Tensegrity-ladders.fig}(e) is the form of Figure \ref{Projections-3.fig}.

\begin{figure}[ht]
\centering 
\includegraphics[scale=0.6]{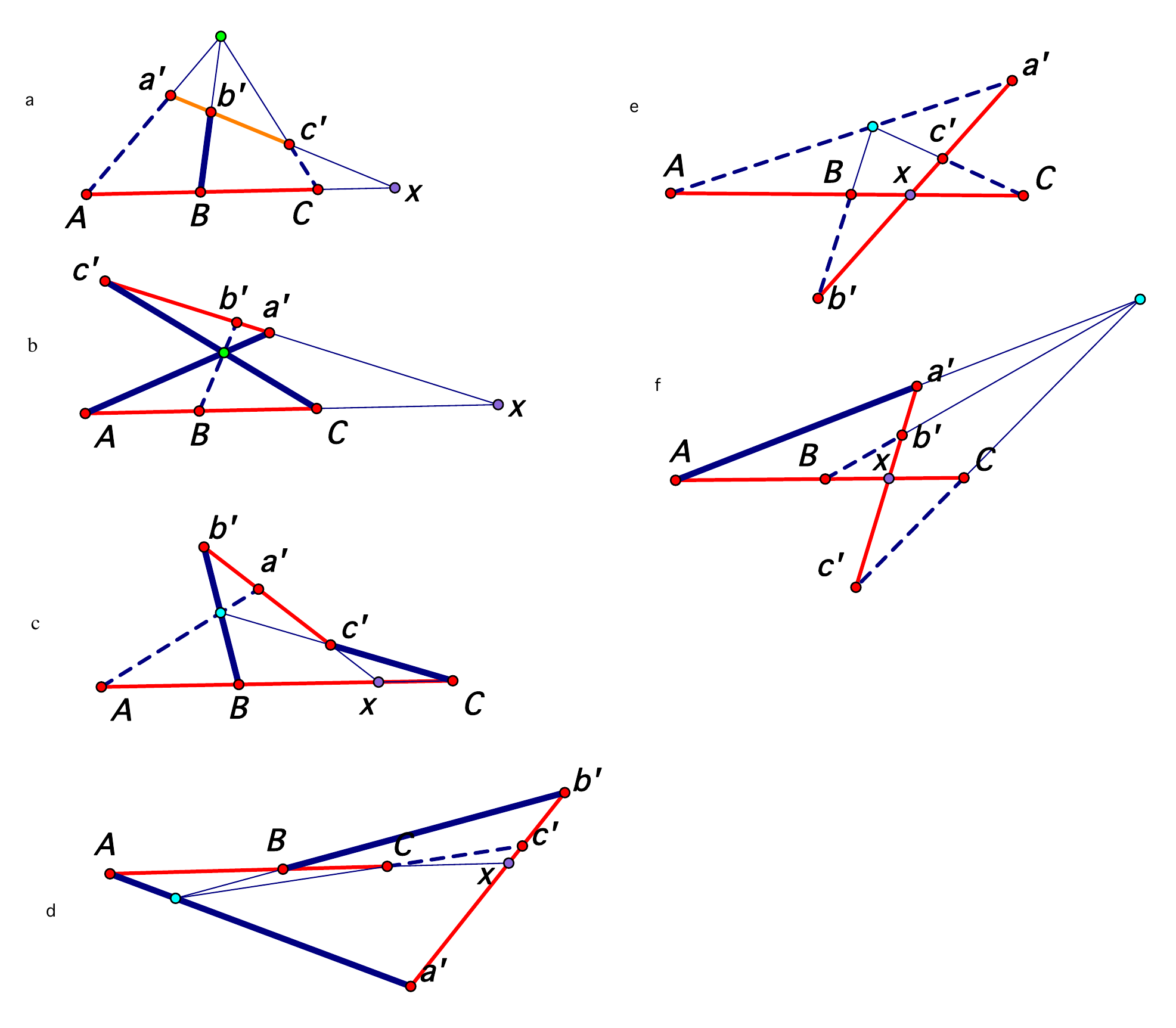}
\captionsetup{labelsep=colon,margin=1.3cm}
\caption{This shows universally rigid tensegrity ladders  where the poles, in red, are in different lines.  As in previous figures, the green points indicate the rung points, and the dark points, labeled $x$ are where the poles intersect.  Cables and struts are indicated by dashed edges and thick edges, respectively. }
\label{Tensegrity-ladders.fig}
\end{figure}
 
 In \cite{iterative} it is shown that a ladder such as in Figure \ref{Tensegrity-ladders.fig} (a) is "two-step" universally rigid.  Indeed, it is possible to show that any of the ladders in Figure \ref{Tensegrity-ladders.fig} are universally rigid, but, of course, they are not super-stable.  It is also possible to show that in each of these cases, they are universally rigid as tensegrities.  In other words when the rung lines meet at a point and the secondary stress signs are as in Figure \ref{Tensegrity-ladders.fig} (cable stresses positive and strut stresses negative), then the ladder is universally rigid.  
 
 It is interesting to show what happens in general when we vary one of the coordinates, $A,B,C,a,b,c$ that define the ladder as above in this section.  
 
 \begin{theorem} Let $\p=(A,B,C,a,b,c)$ be the coordinates that define a ladder $(G,\p)$ in $\RR^1$ where no edge has zero length.  Then the set $\{t|D(A,B,C,a,b,t)\ge0\}$ is either a closed interval of non-zero length, or the complement of an open interval of non-zero length, a closed ray $\{ t\ge  t_0\}$ or $\{ -t\ge  t_0\}$, all of $\RR^1$, or the empty set.
 \end{theorem}
 \begin{proof} Regarding the $A,B,C,a,b$ as fixed, then the discriminant function $D(A,B,C,a,b,t)$ is at most degree two in $t$. 
 If the discriminant is  actually linear
 (or constant)
 in $t$, we are done. 
 Otherwise,  the 
 discriminant is degree two in $t$, then
the statement of the theorem says that the solution to 
 \[
 D(A,B,C,a,b,t)=0
 \]
 cannot be a single point.  This can only happen when the discriminant $DD$ of that quadratic function is zero.   By choosing an appropriate projective transformation, we can assume that $A=-1, B=0, C=1$, and then
 \[
 DD(a,b)=-64b(a+1)(a-b)^3.
 \]
 This is never $0$, except when $B=b=0$ or $a=-1=A$, or $a=b$, which is when one of the edge lengths of $G$ is $0$.\qedsymbol \end{proof} 
 
 \begin{corollary}\label{end-vary} If the discriminate $D(A,B,C,a,b,c)=0$ corresponds to a configuration of the ladder with non-zero edge lengths, then any one of the coordinates $A,B,C,a,b,c$ can be perturbed one way to create a configuration that is super stable with positive discriminant, and perturbed the other way to create a configuration that has discriminant negative.
 \end{corollary}
 \begin{proof} 
 Suppose $D(A,B,C,a,b,c)=0$, and  one varies, say $c$, fixing the others.
From the theorem, $c$ must be one of the 
endpoints of a region of $c$ values where
the disriminant is positive. \qedsymbol
 \end{proof}
 
 It is possible to have configurations $\p$, such that when one variable changes, fixing all the others (with no zero edge lengths), all the configurations are super stable. For example, that happens when the $A,B,a,b$ coordinates form a stretched cycle as in Theorem \ref{thm:stretched}.  On the other hand, a calculation shows that it is not possible to have configuration $\p$ such that for all values of one variable the frameworks are not universally rigid.
 
 \section{Negative Discriminant}\label{Sect:negative-discriminant}
 
 We now know that when the discriminant $D(A,B,C,a,b,c) > 0$, then the corresponding configuration in the line is the orthogonal projection of a ladder in the plane and it is not just  universally rigid, but it is super stable by Theorem \ref{discriminant:thm}.  We next want to show that when $D(A,B,C,a,b,c)< 0$, then the corresponding framework in the line is not universally rigid. 
 To this end, we will first show that if the
discriminant $D(A,B,C,a,b,c) \le 0$, then
the ladder framework in $\RR^1$ is not super stable.

When the ladder in the line is such that one rung contains another in reverse order such as $A<B>b>a>A$ or $A<C>c>a>A$ or $B<C>c>b>B$ or with all the inequalities reversed, we say any of these is a \defn{stretched cycle}.  It is not hard to see that if the ladder $(G,\p)$ has a stretched cycle, then it is super stable, but we can say a bit more:

\begin{theorem}\label{thm:stretched} If a ladder in the line $(G,\p)$ has a stretched cycle, then  $D(A,B,C,a,b,c) > 0$ and $(G,\p)$ is super stable in the line.
\end{theorem}
 
 This can be seen by looking at the definition of the discriminant, fixing the four vertices of the stretched cycle, and then considering the polynomial in terms of the other two variables.  That polynomial is quadratic in each variable and discriminant of each one is negative.  Indeed, when $A,B,b,a$ is a stretched cycle, one can use a projective transformation so that 
 $A=1, B\ge 0, a=-1, b=-B$
 while $C$ and $c$ are arbitrary. Maple verifies that the minimum of the discriminant $D(1,B,C,-1,-b,c)$ is zero when $B\ge 0$.  Using this stress and adding a PSD stress from the one of the triangles implies that $(G,p)$ is superstable.
 Alternatively, one can use an argument similar to Theorem \ref{thm:ladder-nonUR}.
 
\begin{lemma}
\label{lem:replace}
Let $(G,\p)$ be  a bar framework in $\RR^1$, with more than
one edge, that is universally rigid.
Then we can replace each of the bars by one of either a strut or a  cable, and it still will be universally rigid
as a tensegrity framework.
\end{lemma}
\begin{proof}
To outline the basic idea of the proof, let us 
assume, for starters, that $(G,\p)$ is super stable. 
Next, replace each bar that has a positive stress coefficient with a cable and each bar that has a negative
stress with a strut. 
Replace unstressed bars with either a cable or strut,
it does not matter.
The proper stress is a certificate that 
our new tensegrity framework is dimensionally rigid. 
To upgrade from dimensional rigidity to universal rigidity
we just have to demonstrate that there is no
affine motion in $\RR^1$ that preserves the tensegrity 
constraints. In $\RR^1$, this is ensured by the
fact that we have at least one strut and one cable,
(which is insured by our assumed stress and the fact that
we have more than one edge).

Unfortunately, in this lemma, we don't assume super stability, so we
do not have the above stress to work with. 
Fortunately, we can replace the single stress with an
``iterated stress sequence'' as described in~\cite{iterative}.
In particular, 
 Corollary 8.1.2 in  \cite{iterative},  says that
 when a bar framework $(G,\p)$ is universally rigid, there is a sequence of iterated
 PSD stresses generalizing the 
 notion of a stress as defined here. Moreover, 
 once an edge appears with a non-zero coefficient,
 its sign can be maintained in each subsequent iteration level. 
 (See Proof of Theorem 12.1 in \cite{iterative}.)
 As in the super stable case above, we  replace each bar that has a positive coefficient in the iterated
 stress with a cable and each bar that has a negative
coefficient in the iterated stress
with a strut. 
Replace unstressed bars with either a cable or strut. 
It does not matter.
We now have a tensegrity with
an iterated proper PSD stress. 
Again from  Theorem 12.1 in \cite{iterative}, this stress
is a certificate of the dimensional rigidity of our tensegrity 
framework. 
 
Just like above, to upgrade from dimensional rigidity to universal rigidity
we just have to demonstrate that there is no
affine motion in $R^1$ that preserves the tensegrity 
constraints. In $\RR^1$, this is ensured by the
fact that we have at least one strut and one cable. \qedsymbol
\end{proof}

\begin{lemma}
\label{lem:haspsd}
Suppose that $(G,\p)$ is a framework. Suppose that 
we cannot find a framework (possibly in higher dimension) where 
one of its edges is shorter, and the rest of its edges have the
original length. Then $(G,\p)$ must have a non-zero
PSD equilibrium stress.
\end{lemma}
\begin{proof}
Suppose that $(G,\p)$ does not have a non-zero PSD stress matrix.
Then from~\cite{alfakih2011bar}, there must be a framework of
$G$ with an $n-1$ dimensional affine span. 
In this high dimensional 
framework, we can make any edge longer or shorter, at
will, while maintaining all of the other edge lengths. \qedsymbol
\end{proof}

The main result is the following. 

\begin{theorem}\label{thm:ladder-nonUR} Suppose $(G,\p)$ is a framework in $\RR^1$ of the ladder graph
that is universally rigid. Then for any $\epsilon >0$ there is a super stable  framework $(G,\hat{\p})$ of the ladder graph in $\RR^1$ with $||\p-\hat{\p}||<\epsilon$ and the discriminant 
$\hat{\p}$ 
is $D > 0$. 
\end{theorem}
\begin{proof} 
From Lemma \ref{lem:replace} we can 
pick one of the rungs in the ladder 
and repleace it with a cable or strut.  Say it is a cable.  The case of a strut is similar.

Choose a  framework  $(G,\q)$
in $\RR^5$, close to $\p$, such that the two triangles each are not collinear.   Fix all those edges as bar lengths, except one rung length $r$.  Among all those configurations with those fixed lengths, let $\q^*$ be one with the with the smallest length of $r$. Since the bar graph $G$, minus the one rung, is connected, a configuration $\q^*$ with the minimum length of $r$ must exist.  

Furthermore, 
we claim that for $\q$ chosen sufficiently close to $\p$,
there must be a rung-length-minimizing configuration $\q^*$ such $||\p-\q^*||<\epsilon$.
Suppose there was not. 
Then we could choose a 
sequence of $\q(i)$s that converged to $\p$,
and minimize the rung length to get  sequence of $\q(i)^*$s.
By compactness, there would be a 
convergent subsequence with  the limit configuration,
$\bar{\q}^*$.
The configuration $\bar{\q}^*$
must have bar lengths equal to those of $\p$, and rung
length less than or equal to its length in $\p$.
Meanwhile because we have assumed that we cannot
find a $\q^*$ that is within $\epsilon$ of $\p$, then
$\bar{\q}^*$ cannot be congruent to $\p$.
This $\bar{\q}^*$ then would violate the universal rigidity of our tensegrity $\p$.

Because $\q^*$ is an optimum of our minimization problem, then
from Lemma~\ref{lem:haspsd}
$(G,\q^*)$ has a non-zero equilibrium stress $\omega$ with a non-zero PSD stress matrix $\Omega$.  
Note that, since 
${\q}$ is near to $\p$, its rung length cannot be zero, and
so our PSD stress must involve more than one edge.

We next consider the sign pattern of the stress with the PSD stress matrix $\Omega$.  Recall each triangle is not in a line, and thus the vertices of each triangle must span a plane.  Recall that each vertex is of degree $3$, so when the edges adjacent to a vertex are such that no $2$ are collinear, any non-zero stress in one edge is non-zero in the other $2$.
The stresses $\omega_{ij}$ cannot be $0$ on all edges $ij$ of either triangle of $\q^*$, since if that were to happen on one triangle, then all the rung stresses would be $0$, and then all the other stresses on the other triangle would also be zero.    So at least one of the rungs must have a non-zero stress.  

If only one rung has a non-zero stress, one cannot have equilibrium at the vertices of that rung, since the other vertices of the triangles could not be in equilibrium.  
So this cannot happen.

Suppose that there are exactly two rungs with a non-zero stress.  Say the two rungs are $A,a$ and $B,b$.  Then we have a cycle of four edges $(A,a),(a,b),(b,B),(B,A)$, which are the only edges with non-zero stress.
If the signs of that cycle of four edges are $+,-,+.-$, the stress matrix $\Omega$ cannot be PSD. Thus one sign is negative and the other three are positive. 
Thus $\q^*$ has a stretched cycle. $\p$ is nearby, so it too must
have a stretched cycle. 
Thus from Theorem \ref{thm:stretched}
$\p$ itself is super stable and
$D(\p)>0$ and we are done.

In the remaining case, 
we know that all the rungs have a non-zero stress.
In this case, all the vertices of $G$ are adjacent to an edge with a non-zero stress.
Equilibrium at each vertex of that triangle implies that each of the rung edges lie in that plane.
So all the vertices lie in that plane.
So there is no configuration with that stress that spans a higher dimension than $2$.  
That means that the rank of $\Omega$ is $6-3=3$, the maximum.
We already know that the stress matrix 
is PSD.
Since each of the triangles of $\q$ does not lie in a line, the same is true for $\q^*$, so the edges of $\q^*$
are in more than $2$ directions and so are not 
on a conic at infinity.
 Thus $(G,\q^*)$ is super stable
in $\RR^2$.   
So we know that $\q^*$
is a super stable Desargures configuration, where the three rung lines intersect at a point (possibly at infinity).

We could now show directly that the orthogonal projection of $\q^*$ into $\RR^1$ is, in fact, super stable, but we would like to also show that there is a configuration $\hat{\p}$ within $\epsilon$ of $\p$ in the line that has $D(\hat{\p})>0$, which by Theorem \ref{discriminant:thm} shows that $\hat{\p}$ is super stable as well. In order calculate the discriminant, we 
will want an orchard ladder that projects into $\RR^1$. To do that we will set up a limiting process over a sequence of
$\q^*(i)$. 

Recall that in the original optimization, all the lengths of the edges are fixed except one rung. 
And we showed that there is a $\delta > 0$ such that when the starting configuration $\q$ is within $\delta$ of $\p$, there
is a critical (max or min rung length) configuration $\q^*$ 
that stays in the $\epsilon$ neighborhood of $\p$. 
In the optimization, 
the shape of two triangles in $\q$, corresponding to $A,B,C$ and $a,b,c$, can be chosen to be as close to being in a line as we like.  

Let us set up a sequence of $\q(i)$ near to $\p$,
where the two triangles  converge to  line segments, and the quadrilateral not involving the optimized rung has
fixed lengths that cannot be flattened into the line.

The optimization now gives us a sequence $\q^*(i)$, with 
a convergent subsequence with limit $\q^{**}$
and $||\p-\q^{**}||\le \epsilon$. 
It has a two dimensional span.
Its three rung lines intersect at a point (possibly at infinity). 
Its two triangles both lie in lines.
Thus, $(G,\q^{**})$ is a ladder (orchard or step).

The projection of $\q^{**}$ into $\RR^1$ is such that $D(\q^{**}) \ge 0$ from Theorem~\ref{thm:newRatio}.
Applying Corollary \ref{end-vary} we can then wiggle 
this projected framework to get the desired configuration
$\hat{\p}$ with positive discriminant, and thus, also super stable.
It is within $\epsilon$ of $\p$ (for a slightly increased
$\epsilon$).
\qedsymbol 
\end{proof} 

As a final wrap-up, putting all the results together from above, we describe the universal rigidity of the ladder graph in $\RR^1$ completely.  Note that when the pole point $x$ on the line through one pole is at $\infty$, by Equation (\ref{eq: ratio}) that means 
\begin{equation}\label{eq:infty}
\frac{C-A}{B-A}=\frac{c-a}{b-a}. 
\end{equation}

\begin{theorem}\label{cor:wrap-up} For a ladder framework $G(\p)$ in $\RR^1$ with all edges non-zero,
\begin{enumerate}
    \item If $D(\p)>0$, then $(G,\p)$ is super stable. 
    \item If $D(\p)<0$, then $(G,\p)$ is not universally rigid. 
    \item If $D(\p)=0$, and $\mathcal{A} \ne 0$ then $(G,\p)$ is universally rigid, but not super stable.
    \item If $D(\p)=0$, and $\mathcal{A} = 0$, then $(G,\p)$ is not universally rigid and flexes as a step ladder.
\end{enumerate}
\end{theorem}
\begin{proof} 
    \begin{enumerate}
        \item This is Theorem \ref{discriminant:thm}
        \item If $(G,\p)$ were universally rigid, then by Theorem \ref{thm:ladder-nonUR} and Theorem~\ref{thm:stretched}
        $(G,\p)$ 
        can be approximated by frameworks $(G,\hat{\p})$ where $D(\q)> 0$.  Then we would have $D(\p) \ge 0$, a contradiction. 
        \item When $D(\p)=0$ and $\mathcal{A} \ne 0$, then the pole point $x$ is finite, as in Figure \ref{Tensegrity-ladders.fig}, and thus
        the poles of any lift $(G,\q)$ in $\RR^2$ 
        will not  be parallel.
        The result then follows from Lemma \ref{lem:D=0}.
        
        \item If $D(p)=0$ and 
        $\mathcal{A} = 0$, this implies that
        $\mathcal{B}=0$. This means that anly lift is a step ladder
        as in the proof of Theorem~\ref{thm:newRatio}.
    \end{enumerate}
    \qedsymbol
\end{proof}

\section{How it flexes}

We know that when the discriminant D is negative, that it flexes at least into $\RR^3$.  But it does so in an interesting way.  But first an observation:

\begin{theorem}\label{thm:planar-ladders} If a ladder in the plane is not contained in a line, and is not an orchard ladder or a step ladder, then it is rigid in the plane, and flexible up into $\RR^3$ and then back down again into the plane.  
\end{theorem}
\begin{figure}[ht]
\centering 
\includegraphics[scale=0.4]{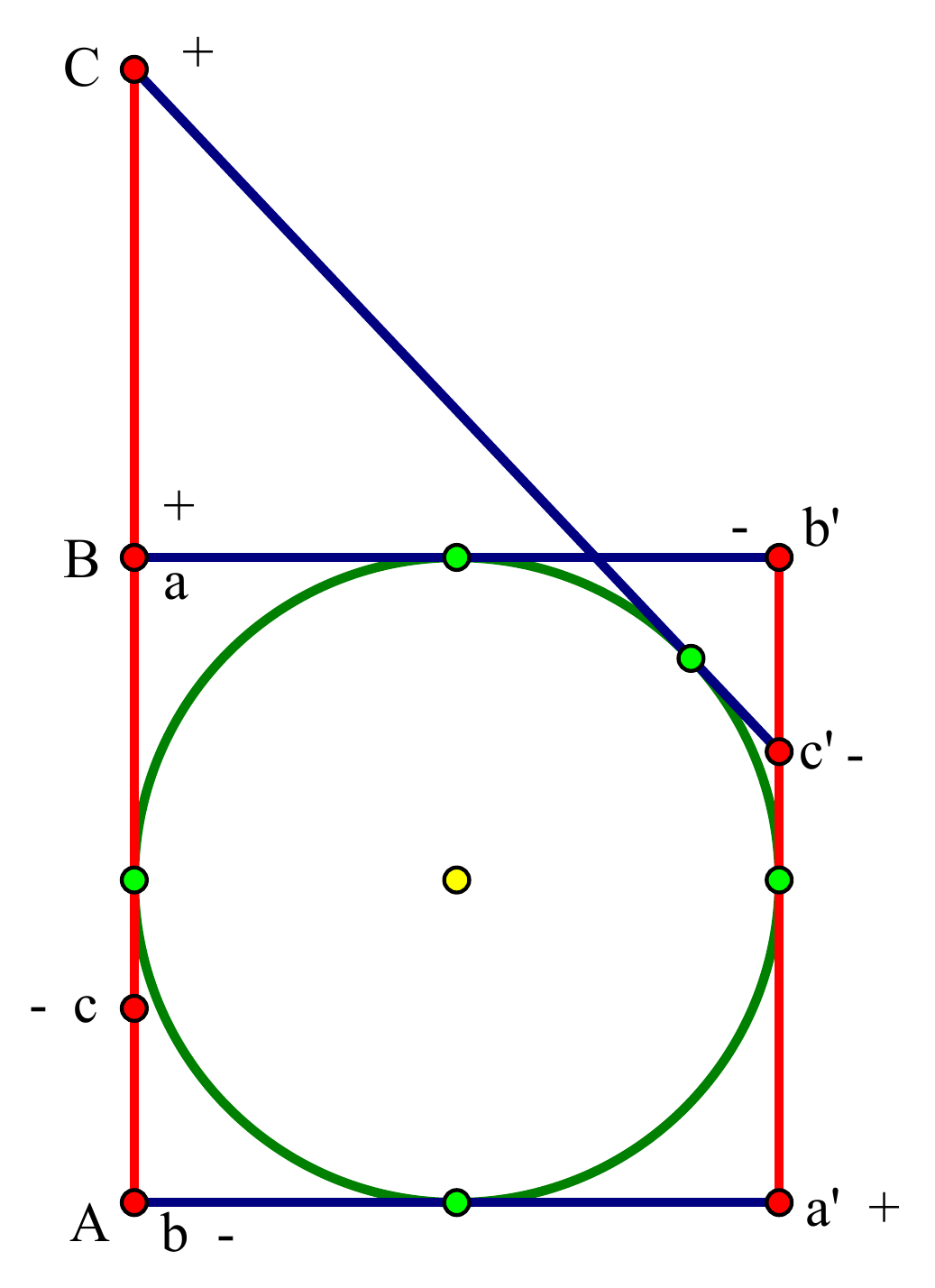}
\captionsetup{labelsep=colon,margin=1.3cm}
\caption{This shows a ladder, not an orchard ladder, that is rigid in the plane, but flexes in $\RR^3$ and then back down into the line. Notice that some of the vertices overlap, but no adjacent vertices overlap. Notice also that there is a $+$ or $-$ next to each of the vertices.  The map from hyperboloid in $\RR^3$ to $\RR^2$ is two-to-one except for points on the conic, and the sign indicates whether the point is above or below the limiting plane as it converges to the conic, in this case a circle.   The green points are the intersection of each line through each edge of the framework with that circle in the center of the hyperboloid.   When the whole framework flexes into $\RR^1$ the green points all fold into a single point.}
\label{Flexible-ladders.fig}
\end{figure}
\begin{proof} If the three rung lines intersect at a point, then the ladder is universally rigid as an orchard ladder, or it is a step ladder and flexes in the plane. Otherwise all the five lines in the plane are such that no three meet at a point since all the member lengths are non-zero.  Thus by duality in the projective plane, the five lines are tangent to a unique conic.  But then the lines tangent to that conic are the limiting flex of configurations of some ruled hyperbolic surface in $\RR^3$, which flexes as indicated in Figure \ref{hyperboloid.fig}.  Since the rung lines do not meet at a point (including a point at infinity), it is possible to show that the framework, in the plane, is "prestress stable" which implies it is locally rigid. (See \cite{book} for definitions and results.)\qedsymbol
\end{proof}

It is easy to see from the definition of the flex in page 99 in \cite{book} that the flex of a doubly ruled hyperboloid will flex down to a line if and only if two of the eigenvalues of the symmetric matrix, that describes the conic at infinity, are equal.  When a ladder is in $\RR^3$ and the lines through the poles are disjoint, then the lines through the rungs determine a ruled hyperboloid surface.  Then that surface, along with its embedded ladder, flexes into a plane, and the resulting rulings are all tangent to a circle.  You can see the hyperboloid in Figure \ref{hyperboloid.fig} below. 

\begin{figure}[ht]
\centering 
\includegraphics[scale=0.3]{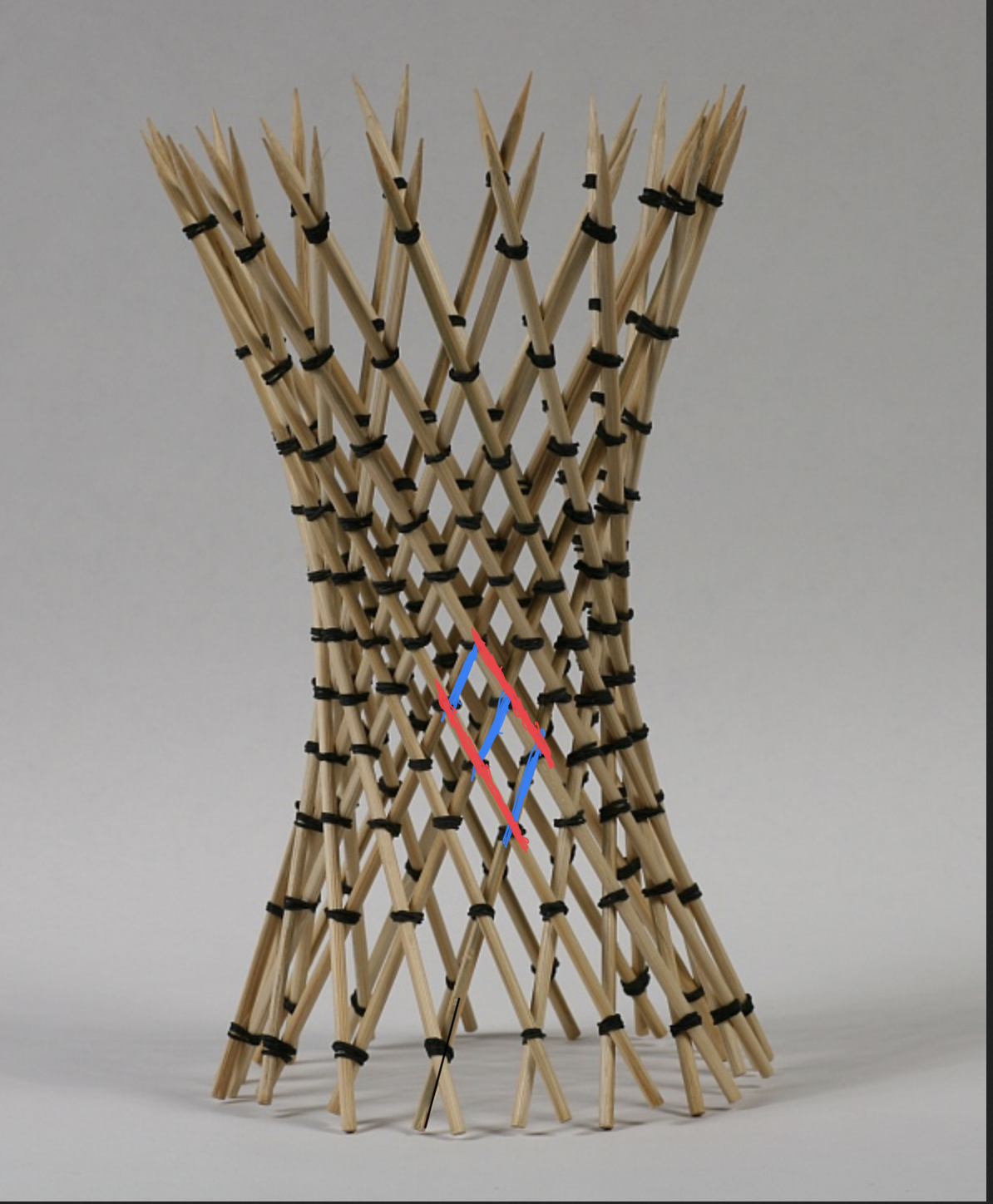}
\captionsetup{labelsep=colon,margin=1.3cm}
\caption{This shows a hyperboloid with circular symmetry that flexes into the plane around a circle, preserving lengths along the rulings. A typical ladder is shown in red and blue.}
\label{hyperboloid.fig}
\end{figure}
A doubly ruled hyperbolic surface in $\RR^3$ has its rulings correspond to a three-by-three symmetric matrix.  If the eigenvalues of that matrix are all different (and necessarily not all the same sign), then the hyperboloid ruling lines flex, and project orthogonally, into $\RR^2$ two ways.  One way is such that the ruling lines are all tangent to an ellipse, and the other way is such that the ruling lines are all tangent to an hyperbola. 
When two of the eigenvalues are equal, we have the situation of Figure \ref{hyperboloid.fig}.
When there are only two non-zero eigenvalues, necessarily of opposite sign, then the conic at infinity is two lines, and the hyperboloid is a parabolic hyperboloid. The ruling lines flex and project onto the lines tangent to a parabola for both ends of the flex.


\section{Addendum}

Here we show that the eleven cable-strut patterns of the two projectivity classes of super stable Desargues graphs are all the possible sign patterns and projectivity classes.

\begin{lemma}\label{sign:thm} If $(G,\p)$ is a super stable tensegrity framework and $p_0$ is a vertex with exactly three neighbors $p_1, p_2, p_3$, no three in a line, then the configuration must be one of the types in Figure \ref{local-stress.fig}.  In case (a), the each pair of successive cables must be less than $180$ degrees apart.  In cases (b) and (c), the direction of  $p_0,p_3$ must be between the  $p_0,p_1$ and $p_0,p_2$ directions.  In case (b), $p_3$ must not be in the triangle formed by $p_0,p_1,p_2$.  In case (c), $p_3$ must  be in the triangle formed by $p_0,p_1,p_2$.
\end{lemma}
\begin{figure}[ht]
\centering 
\includegraphics[scale=0.5]{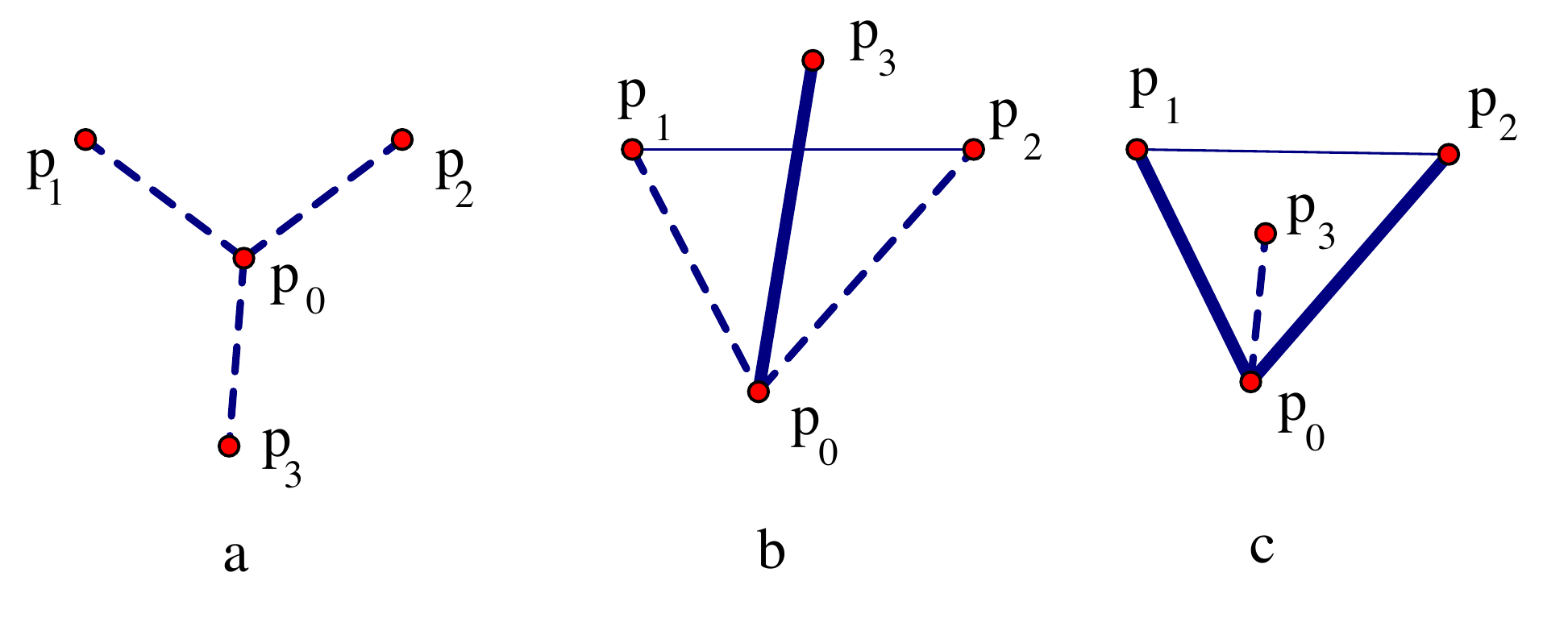}
\captionsetup{labelsep=colon,margin=1.3cm}
\caption{This shows the three cases of super-stable configurations with three edges at a vertex $p_0$ with the cable/strut designation.}
\label{local-stress.fig}
\end{figure}

\begin{proof} Clearly all four vertices must be in a plane, by the equilibrium condition.   When all the edges are cables in case (a), the $180$ degree condition is clear since the stresses are all the same sign.

When two of the edges are cables and one a strut are as in as in case (b) the direction is clear from the signs of the stresses.  If $p_3$ is in the $p_0,p_1,p_2$ triangle, then when $p_0$ is rotated up out of the plane, fixing  $p_1,p_2,p_3, ||p_1-p_0||$, and $||p_2-p_0||$, then $||p_3-p_0||$ increases contradicting the super stability of $(G,\p)$.  

When two of the edges are cables and one a strut are as in as in case (b) the direction is clear from the signs of the stresses.  If $p_3$ is not in the $p_0,p_1,p_2$ triangle, then when $p_0$ is rotated up out of the plane, fixing  $p_1,p_2,p_3, ||p_1-p_0||$, and $||p_2-p_0||$, then $||p_3-p_0||$ decreases contradicting the super stability of $(G,\p)$. 

When all the edges are struts, $(G,\p)$ is clearly not super stable since $p_0$ can be sent to Timbuktu.
\qedsymbol
\end{proof}

In Figure \ref{slicers.fig} there are pictures of two representations of the Desargues as tensegrities in the plane that are not projectively equivalent, along with lines that cut through some of the cable and strut segments.  When a projectivity sends a line to infinity, and that line intersects a cable or strut, in the image tensegrity, the roll of it being a cable or strut is reversed.  When the line does not intersect the cable or strut, it does not change.  Up to symmetries of the abstract graph, the lines indicated are all the possibilities.  

\begin{figure}[ht]
\centering 
\includegraphics[scale=0.4]{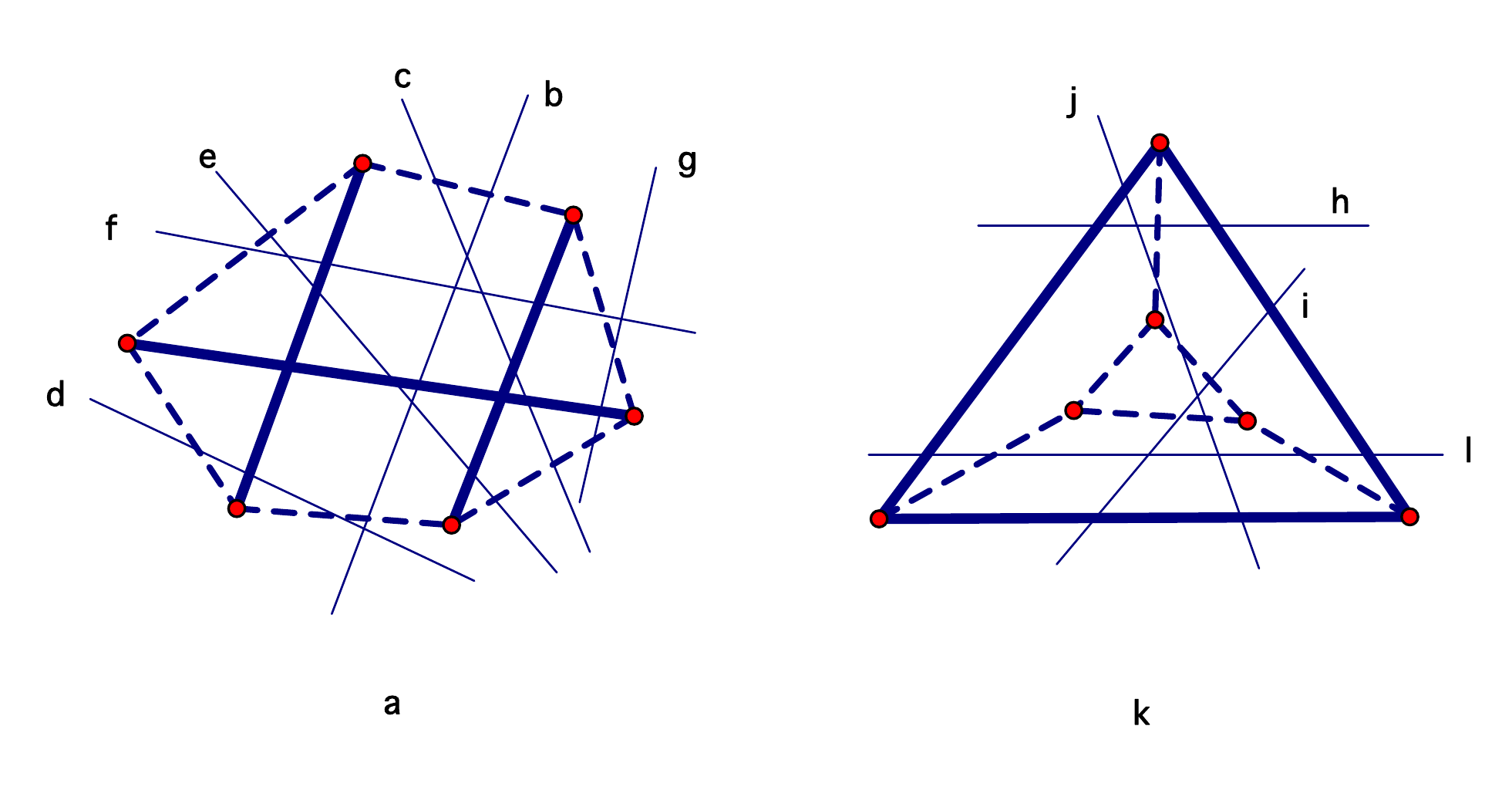}
\captionsetup{labelsep=colon,margin=1.3cm}
\caption{Here are two configurations of the Desargues graph that are projectively inequivalent. Also included are lines that intersect the edges of each framework such that when those lines are sent to infinity by a projective map of the plane, then the configuration becomes the corresponding configuration in Figure \ref{Tensegrity.fig}.}
\label{slicers.fig}
\end{figure}

 Using the geometric conditions of Lemma \ref{sign:thm} one can find  the geometric shape of the tensegrities in Figure \ref{Tensegrity.fig}. 
 
 \begin{theorem} The sign pattern and configurations of Figure \ref{Tensegrity.fig} are all the possibilities for super stable tensegrities for the Desargues graph, no three vertices in a line.
 \end{theorem}
 \begin{proof} One can see that if there is a super stable tensegrity for the Desargues configuration, where one of the triangles is all struts, then either configuration $g$ or $h$ of Figure \ref{Tensegrity.fig} is the sign pattern. 
 
 It is also easy to see that if any tensegrity is super stable (or universally rigid for that matter), then the cable edges of $G$ connect all the vertices of $G$.  If not, seperate components can be positioned arbitrarily far apart, satisfying the cable/strut conditions. 
 
 If one of the triangles in the Desargues graph $G$ is cable-cable-strut, then sending a line through the two cables to infinity by a projective transformation, we can create another configuration with a triangle of all struts.
 
 So we are left with the case that there is a triangle consisting of all cables.  Not all the rung edges can be struts, since that would create a tensegrity where the cable graph is not connected.  Figure \ref{cable-triangle.fig} shows the three cases when the three rung edges are all cables (top row), two cables and a strut (second row), and one cable and two struts (bottom row).  The left column corresponds to examples in Figure \ref{Tensegrity.fig}, while the two examples in the right column are not PSD, since their cable graphs are not connected.  The right column arises from the positioning of the three vertices of the other triangle and from the conditions of  Lemma \ref{sign:thm}.\qedsymbol
\end{proof}
\begin{figure}[ht] 
\centering 
\includegraphics[scale=0.3]{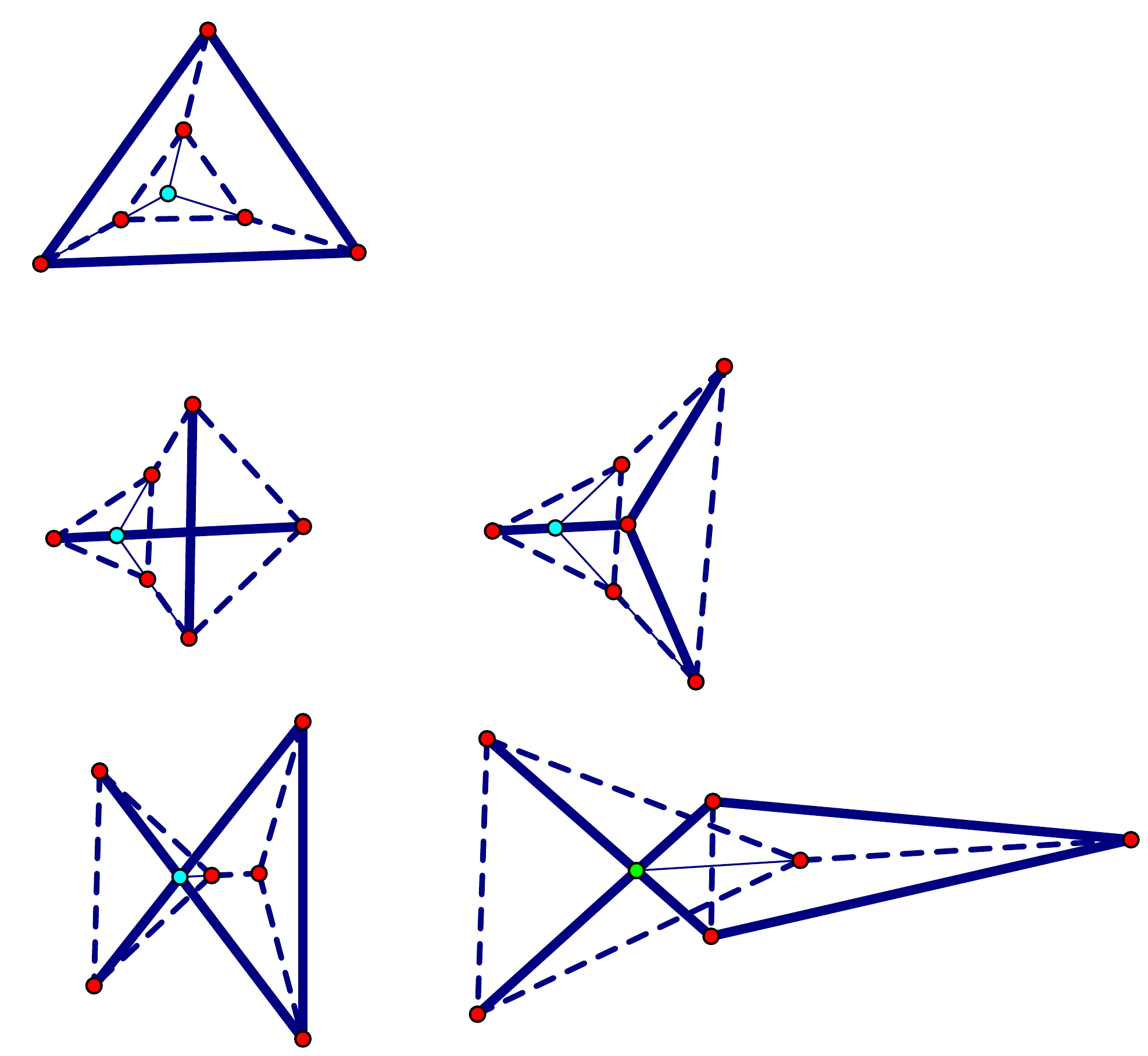}
\captionsetup{labelsep=colon,margin=1.3cm}
\caption{When one of the triangles in the graph $G$ is all cables, this is a list of possible sign patterns that satisfy the conditions of Lemma \ref{sign:thm}.  Only the left column of tensegrities is such that they have a PSD stress matrix, and they are examples of the tensegrities $k, h, l$ from Figure \ref{Tensegrity.fig} respectively.}
\label{cable-triangle.fig}
\end{figure}

If one of the triangles in the Desargues graph $G$ is cable-cable-strut, then sending a line through the two cables to infinity by a projective transformation, we can create another configuration with a triangle of all struts.

\bibliographystyle{abbrvnat}
\bibliography{icms.bib} 

\end{document}